\documentclass[11pt, oneside]{amsart} 

\usepackage[british,english]{babel} 
\usepackage{graphics,color,pgf}
\usepackage{epsfig}
\usepackage[ansinew]{inputenc}
\usepackage[all]{xy}
\usepackage{hyperref}
\newdir{ >}{!/8pt/@{}*@{>}}
\usepackage{amssymb, amsmath,amsthm, mathtools, amscd}
\usepackage{mathrsfs}
\usepackage{stmaryrd}
\usepackage[margin=1.5in]{geometry}

\makeatletter
\newtheorem*{rep@theorem}{\rep@title}
\newcommand{\newreptheorem}[2]{%
\newenvironment{rep#1}[1]{%
 \def\rep@title{#2 \ref{##1}}%
 \begin{rep@theorem}}%
 {\end{rep@theorem}}}
\makeatother

\newtheorem{thmx}{Theorem}

\theoremstyle{plain}
\newtheorem{teor}{Theorem}[section]
\newreptheorem{teor}{Theorem}  
\newtheorem{lem}[teor]{Lemma}

\newtheorem{prop}[teor]{Proposition}
\newreptheorem{prop}{Proposition}  

\newtheorem{setup}[teor]{Setup}

\theoremstyle{definition}
\newtheorem{deft}[teor]{Definition}

\theoremstyle{remark}
\newtheorem{oss}[teor]{Remark}

\DeclareMathOperator\upL{\textup{L}}

\DeclareMathOperator\bbC{\mathbb{C}}

\DeclareMathOperator\bbH{\mathbb{H}}

\DeclareMathOperator\bbN{\mathbb{N}}

\DeclareMathOperator\bbP{\mathbb{P}}

\DeclareMathOperator\bbR{\mathbb{R}}
\DeclareMathOperator\bbS{\mathbb{S}}

\begin{document}

\title[Trivializability for rank-one cocycles]{On the trivializability of rank-one cocycles with an invariant field of projective measures}

\author[A. Savini]{A. Savini}
\address{Dipartimento di Matematica, Politecnico di Milano, Piazza Leonardo da Vinci 32, 20133, Milano, Italia}
\email{alessio.savini@polimi.it}

\date{\today.\ \copyright{\ The author was partially supported by the SNSF grant no. 200020-192216.}}

\begin{abstract}
Let $G$ be $\textup{SO}^\circ(n,1)$ for $n \geq 3$ and consider a lattice $\Gamma < G$. Given a standard Borel probability $\Gamma$-space $(\Omega,\mu)$, consider a measurable cocycle $\sigma:\Gamma \times \Omega \rightarrow \mathbf{H}(\kappa)$, where $\mathbf{H}$ is a connected algebraic $\kappa$-group over a local field $\kappa$. Under the assumption of compatibility between $G$ and the pair $(\mathbf{H},\kappa)$, we show that if $\sigma$ admits an equivariant field of probability measures on a suitable projective space, then $\sigma$ is trivializable. 

An analogous result holds in the complex hyperbolic case. 
\end{abstract}
  
\maketitle

\section{Introduction}

A lattice $\Gamma$ in a locally compact second countable group $G$ is \emph{rigid} if, roughly speaking, its isomorphism class boils down to its conjugacy class. The pioneering works by Mostow \cite{mostow68:articolo,Most73} and Prasad \cite{prasad:articolo} showed that any lattice in a semi-simple Lie group without factors either compact or locally isomorphic to $\textup{SL}(2,\bbR)$ are rigid. Later Margulis \cite{margulis:super} strengthened Mostow's Theorem by showing that any unbounded Zariski dense representation of a higher rank irreducible lattice into an adjoint semi-simple Lie group can be actually extended to the ambient group. Such phenomenon, called \emph{superrigidity}, was exploited by Margulis to give an arithmeticity criterion for higher rank irreducible lattices. 

After those outstanding results, the interest of the mathematical community in rigidity of lattices grew rapidly and nowadays rigidity has become an independent field of research. For instance, the functorial characterization of bounded cohomology due to Burger and Monod \cite{BM99,burger2:articolo} sheds a new light on the theory of \emph{maximal representations}, where the word maximal usually refers to the maximality of a numerical invariant defined in terms of the bounded cohomology of the groups involved. Maximal representations have been widely studied in the case of both real and complex hyperbolic lattices by a long list of authors \cite{iozzi02:articolo,BIW1,bucher2:articolo,Pozzetti,BBIborel}.

More recently, Bader, Fisher, Miller and Stover \cite{BFMS21,BFMS:Unpub} have exploited the theory of algebraic representability of ergodic actions \cite{BF22} to obtain two superrigidity results for Zariski dense representations of either real or complex hyperbolic lattices. In those cases the additional hypothesis is the existence of suitable invariant probability measure on a given projective bundle. Following the spirit of Margulis' Arithmeticity Theorem, the authors applied their superrigidity statements to characterize a family of arithmetic lattices, namely those admitting an infinite family of maximal totally geodesic submanifolds. 

Moving out from the world of representations, an analogous result to Margulis' Superrigidity Theorem was formulated by Zimmer \cite{zimmer:annals} in the context of \emph{measurable cocycles}. A measurable cocycle is Borel measurable map that can be thought of as a twisted representation, where the twist depends on some parameter space. Zimmer proved that any ergodic Zariski dense measurable cocycle associated to a higher rank irreducible lattice must be \emph{superrigid} (here we will say \emph{trivializable}), that is it can be actually untwisted to a representation that depends no more on the additional parameter space. Zimmer studied such superrigidity phenomenon because measurable cocycles appear naturally in the context of dynamical systems, for instance in the case of either measure equivalence or orbit equivalence. 

After the seminal work by Zimmer, several results about trivializability of cocycles appeared in literature. Monod-Shalom \cite{MonShal0} proved a superrigidity statement when the target is the isometry group of a generalized negatively curved space. Trivializability of cocycles were exploited to prove either orbit or measure equivalence rigidity theorems by Furman \cite{furman99,furman99:2} for higher rank lattices and by Bader-Furman-Sauer \cite{sauer:articolo} for real hyperbolic lattices with an additional integrability condition. Following the spirit of maximal representations, the author, together with Moraschini and Sarti, has recently introduced the notion of \emph{maximal cocycles} \cite{moraschini:savini,moraschini:savini:2,SS21,sarti:savini,sarti:savini:2,Savini2020,savini3:articolo} and proved several rigidity results analogous to the ones valid for representations. 

In this manuscript we want to prove two superrigidity results for cocycles similar to the ones proved by Bader, Fisher, Miller and Stover for representations. As suggested by the authors themselves in the introduction of \cite{BFMS21}, they were already suspecting such results to be true, at least in the real case. Here we want to prove those statements explicitly and to underline the main differences with respect to the case of representations. An additional purpose of the author is to stress the importance of algebraic representability of ergodic actions in the proof of such theorems. 

The first superrigidity statement is about real hyperbolic lattices. For the technical assumption about compatibility we refer the reader to Section \ref{sec:compatibility} for a precise definition. 

\begin{thmx}\label{teor:real:case}
Let $G=\textup{SO}^\circ(n,1)$ for $n \geq 3$ and consider a lattice $\Gamma < G$. Consider a simple non-compact group $W < G$ and an ergodic standard Borel probability $\Gamma$-space $(\Omega,\mu)$. Suppose that $\kappa$ is a local field and $\mathbf{H}$ is a connected algebraic $\kappa$-group such that $(\mathbf{H},\kappa)$ is compatible with $G$. Denote by $H=\mathbf{H}(\kappa)$ the $\kappa$-points of $\mathbf{H}$. Consider a measurable cocycle $\sigma:\Gamma \times \Omega \rightarrow H$ which is Zariski dense and has unbounded image. Suppose additionally that there exist an irreducible representation of $\mathbf{H}$ on a $\kappa$-vector space $V$ and a probability-valued measurable map $\Phi:G \times \Omega \rightarrow \mathcal{M}^1(\mathbb{P}(V))$ which is $\sigma$-equivariant and $W$-invariant with respect to the first variable. Then $\sigma$ is trivializable.
\end{thmx}

We have a similar statement that holds for complex hyperbolic lattices. 

\begin{thmx}\label{teor:complex:case}
Let $G=\textup{SU}(n,1)$ for $n \geq 2$ and consider a lattice $\Gamma < G$. Consider a simple non-compact group $W < G$ and an ergodic standard Borel probability $\Gamma$-space $(\Omega,\mu)$. Suppose that $\kappa$ is a local field and $\mathbf{H}$ is a connected algebraic $\kappa$-group. Denote by $H=\mathbf{H}(\kappa)$ the $\kappa$-points of $\mathbf{H}$. Consider a measurable cocycle $\sigma:\Gamma \times \Omega \rightarrow H$ which is Zariski dense and has unbounded image. Suppose additionally that there exist an irreducible representation of $\mathbf{H}$ on a $\kappa$-vector space $V$ and a probability-valued measurable map $\Phi:G \times \Omega \rightarrow \mathcal{M}^1(\mathbb{P}(V))$ which is $\sigma$-equivariant and $W$-invariant with respect to the first variable. If we have that either
\begin{enumerate}
\item the pair $(\mathbf{H},\kappa)$ is compatible, 
\item the field is given by $\kappa=\bbR$ and $\mathbf{H}(\kappa)=\textup{PU}(n,1)$,
\end{enumerate}
then $\sigma$ is trivializable.
\end{thmx}

The proof of Theorem \ref{teor:real:case} is similar to the one given for representations, whereas the case of Theorem \ref{teor:complex:case} requires more work, since one needs to study carefully the slices of an equivariant measurable map between boundaries which appears naturally in the proof. In this context, we first apply the theory by Pozzetti \cite{Pozzetti} of chain preserving Borel maps to prove that the slices are rational. Then we apply the smoothness of the action on the space of rational maps, already used by Zimmer \cite{zimmer:libro} and the author together with Sarti \cite{SS21}, to get the statement. 

We conclude the paper by proving a partial existence result for the measurable map required by both theorems when the lattice admits on its associated homogeneous space an infinite sequence of probability measures weak-$^\ast$ converging to the Haar one. We finally relate Theorem \ref{teor:real:case} and Theorem \ref{teor:complex:case} to the current literature about maximal cocycles. For instance, we can view them as a way to weaken the hypothesis of integrability in the theorems by either Fisher-Hitchman \cite{fisher:hitchman} or Bader-Furman-Sauer \cite{sauer:articolo}.

\subsection*{Acknowlegdements} I would like to thank the anonymous referee for the precious suggestions and comments which allowed me to improve the quality of this paper.

\subsection*{Plan of the paper} The first part of the manuscript is devoted to recall all the definitions that we need throughout the paper. In Section \ref{sec:hyperbolic:groups} we review the basics about the groups $\textup{SO}^\circ(n,1)$ and $\textup{SU}(n,1)$, their boundaries and we conclude with the notion of compatibility. Section \ref{sec:ergodic:actions} is devoted mainly to actions of algebraic groups on algebraic varieties. In Section \ref{sec:smoothness} we give the definition of smooth action, in Section \ref{sec:ergodicity} we move to the ones of ergodic and metrically ergodic action. We conclude with Section \ref{sec:representability:actions} recalling the notion of algebraic representability of an ergodic action.  In Section \ref{sec:boundary:theory} we give some details about measurable cocycles theory and about boundary theory. 

The proofs of Theorems \ref{teor:real:case} and \ref{teor:complex:case} can be found in Section \ref{sec:real:case} and Section \ref{sec:complex:case}, respectively. In Section \ref{sec:existence:map} we investigate about the existence of the measurable map required by the hypothesis of Theorems \ref{teor:real:case} and \ref{teor:complex:case}. We conclude with some comments and remarks about the relation between our results and the current literature. 

\section{The groups $\textup{SO}^\circ(n,1)$ and $\textup{SU}(n,1)$} \label{sec:hyperbolic:groups}

\subsection{Totally geodesic subspaces and chains}\label{sec:standard:subgroups}

In this section we recall the main properties of the group $\textup{SO}^\circ(n,1)$ and $\textup{SU}(n,1)$. We mainly refer to \cite[Section 2, Section 3]{BFMS:Unpub} for a more detailed description. 

We define the group $\textup{U}(n,1)$ as the group of matrices in $\textup{GL}(n+1,\bbC)$ preserving the Hermitian form 
\begin{equation}\label{eq:hermitian:form}
\widetilde{q}_c:\bbC^{n+1} \rightarrow \bbR \ , \ \ \widetilde{q}_c(x_1,\cdots,x_{n+1}):=\sum_{i=1}^n |x_i|^2 - |x_{n+1}|^2 \ .
\end{equation}
The group $\textup{SU}(n,1)$ is the subgroup of $\textup{U}(n,1)$ whose elements are matrices having determinant equal to one. We are going to write $G_c=\textup{SU}(n,1)$. 

If we restrict the form $\widetilde{q}_c$ to $\bbR^{n+1}$, we obtain a quadratic form $\widetilde{q}_r$ over the reals. We can define analogously the group $\textup{O}(n,1)$ as the subgroup of $\textup{GL}(n+1,\bbR)$ preserving $\widetilde{q}_r$. In a similar way, we can define $\textup{SO}(n,1)$ as the subgroup of $\textup{O}(n,1)$ whose matrices have determinant one. Notice that $\textup{SU}(n,1)$ is connected, whereas $\textup{SO}(n,1)$ has two connected components. We denote by $G_r=\textup{SO}^\circ(n,1)$ the connected component of the identity. By the definition, it is clear that $G_r < G_c$. 

By applying a linear change of coordinates to Equation \eqref{eq:hermitian:form}, one can rewrite the form $\widetilde{q}_c$ (and consequently the form $\widetilde{q}_r$) as follows
$$
q_c(x_1,\cdots,x_{n+1})=2\textup{Re}(x_1\overline{x}_{n+1})+\sum_{i=2}^n |x_i|^2 \ .
$$
Let $\mathcal{E}=\{e_1, \cdots , e_{n+1}\}$ be the canonical basis of $\bbC^{n+1}$. We have that the line generated by the vector $e_1 - e_{n+1}$ is negative definite and the stabilizers of such a line are given by
$$
K_c=\textup{Stab}_{G_c}(e_1-e_{n+1})\cong \textup{U}(n) \ , \ \ K_r=\textup{Stab}_{G_r}(e_1-e_{n+1})=\textup{SO}(n) \ .
$$
The group $K_c$ (respectively $K_r$) is a maximal compact subgroup of $G_c$ (respectively $G_r$) and we can identify the quotient $K_c\backslash G_c \cong \bbH^n_{\bbC}$ with the \emph{complex hyperbolic space} (respectively $K_r\backslash G_r\cong \bbH^n_{\bbR}$ with the \emph{real hyperbolic space}). 

Following \cite{BFMS:Unpub}, for $1 \leq m \leq n$ we denote by $W_c^m\leq G_c$ the subgroup fixing the vectors $e_{m+1},\cdots,e_n$ and similarly we consider $W_r^m = W_c^m \cap G_r$. It is immediate to verify that we have the following isomorphisms 
$$
W_c^m \cong \textup{SU}(m,1) \ , \ \ W_r^m \cong \textup{SO}^\circ(m,1) \ .
$$

For $1 \leq m \leq n$, the subgroups $W_r^m<G_r$ and $W_r^m,W_c^m <G_c$ are called \emph{standard} subgroups. By \cite[Proposition 2.4]{BFMS:Unpub}, the standard subgroups of both $G_r$ and $G_c$ are non-compact, connected, almost simple, closed and generated by unipotent elements. Additionally any other non-compact, connected, almost simple, closed subgroup of either $G_r$ or $G_c$ generated by unipotent elements must be conjugate to a unique standard subgroup. The importance of standard subgroups relies on the fact that they completely characterize totally geodesic subspaces of either $\bbH^n_{\bbR}$ or $\bbH^n_{\bbC}$. More precisely by \cite[Proposition 2.3]{BFMS:Unpub} any totally geodesic submanifold of $\bbH^n_{\bbC}=K_c \backslash G_c$ (respectively of $\bbH^n_{\bbR}=K_r \backslash G_r$) with real dimension greater than or equal to $2$ must be of the form $K\backslash KWg$, where $K$ is the maximal compact subgroup $K_c$ (respectively $K_r$), $W$ is a standard subgroup and $g$ is an element of the group $G_c$ (respectively $G_r$). 

We conclude this short overview about hyperbolic spaces talking about the \emph{chain geometry} of the boundary at infinity $\partial_\infty \bbH^n_{\bbC}$. The latter can be identified with the boundary sphere $\bbS^{2n-1}$ in the ball model. We call a $m$-\emph{plane} of $\bbH^n_{\bbC}$ a totally geodesic holomorphic subspace associated to a standard subgroup conjugated to $W^m_c<G_c$. A $1$-plane is a \emph{complex geodesic}. Similarly, a $k$-\emph{chain} is the boundary of a $k$-plane viewed as a subset of the boundary at infinity $\partial_\infty \bbH^n_{\bbC}$. When $k=1$, we simply refer to \emph{chains}. Two distinct points $\xi,\eta \in \partial_\infty \bbH^n_{\bbC}$ determine a unique chain $C_{\xi,\eta}$ and two chains can be either disjoint or meet exactly in one point. We denote by $\mathcal{C}_1$ the (homogeneous) space parametrizing the pair $(C,\eta)$, where $C$ is a chain and $\eta \in C$. If we denote by $(\bbH^n_{\bbC})^{(2)}$ the set of pairs of distinct points, it should be clear that there exists a natural map
$$
\pi:(\bbH^n_{\bbC})^{(2)} \rightarrow \mathcal{C}_1 \ , \ \  \pi(\xi,\eta)=(C_{\xi,\eta},\xi) \ .
$$

\begin{deft}\label{def:chain:preserving}
Let $\varphi:\partial_\infty \bbH^n_{\bbC} \rightarrow \partial_\infty \bbH^n_{\bbC}$ be a Borel measurable map. We say that the map $\varphi$ is \emph{chain preserving} if it is essentially injective (and hence it induces a well-defined measurable map $\varphi^{(2)}$ on $(\bbH^n_{\bbC})^{(2)}$) and there exists map $\psi:\mathcal{C}_1 \rightarrow \mathcal{C}_1$ such that the following diagram commutes
$$
\xymatrix{
(\bbH^n_{\bbC})^{(2)} \ar[rr]^{\varphi^{(2)}} \ar[d]^\pi && (\bbH^n_{\bbC})^{(2)} \ar[d]^\pi \\
\mathcal{C}_1 \ar[rr]^\psi && \mathcal{C}_1 \ . 
}
$$ 
\end{deft}

Notice that the previous definition is a slight variation of the original definition of chain preserving map given by Pozzetti \cite{Pozzetti}. Our definition implies that the map $\varphi$ sends chains to chains and its restriction to a chain is well-defined for almost every chain, thus it implies that $\varphi$ is chain preserving in the sense of Pozzetti.  

\subsection{Parabolic subgroups and compatibility}\label{sec:compatibility}
In this section we are going to recall most of the subgroups of $\textup{SU}(n,1)$ that we will need in the proof of our Theorem \ref{teor:complex:case}. We are going to recall also the notion of compatibility that will be crucial for all along the paper. We mainly refer the reader to \cite{BFMS21,BFMS:Unpub} for a more detailed discussion. 

We retain the notation of Section \ref{sec:standard:subgroups}. We denote by $P<G_c$ the stabilizer of the line spanned by $e_1$. It is a \emph{(minimal) parabolic subgroup}. Since the rank of the group $G_c$ is equal to one, all the (minimal) parabolic subgroups are actually conjugate. 

By considering the Siegel model of $\bbH^n_{\bbC}$, we can suppose that $e_1$ coincides with the point at infinity. We denote by $MA$ the stabilizer of the pair of points $(\infty,0)$ in the boundary of the Siegel model. Here $M<K_c$ is the subgroup of $K_c$ preserving such pair, whereas $A<P$ is a maximal abelian subgroup isomorphic to $\bbR^\ast$. Notice that we have an identification $$MA \backslash G_c \cong (\partial_\infty \bbH^n_{\bbC})^{(2)} \ . $$ 

If we define $U<P$ as the unipotent radical of $P$, it is well-known that it is isormorphic to the real $(2n-1)$-dimensional Heisenberg group. In particular, it is a two step nilpotent group with center $Z \cong \bbR$. Let $D<G_c$ the stabilizer of the plane generated by $e_1,e_{n+1}$. By \cite[Lemma 2.4]{BFMS:Unpub} we have that 
$$
P \cap D \cong MAZ \ ,
$$
and we have an identification $$ P \cap D \backslash G_c \cong \mathcal{C}_1 \ , $$ where $\mathcal{C}_1$ is the space of pointed chains introduced in Section \ref{sec:standard:subgroups}. 

Let $C<M$ be the subgroup of scalar matrices (isomorphic to the group of the roots of unity). By \cite[Proposition 2.12]{BFMS:Unpub} any continuous homormophism $\tau:P \rightarrow \overline{P}:=P/CZ$ with one-dimensional kernel must be surjective, $\ker(\tau)=CZ$ and, up to precomposing with an inner automorphism of $P$, we can suppose that $\tau(P\cap D)=\theta(P \cap D)$, where $\theta:P \rightarrow \overline{P}$ is the standard quotient projection. We will exploit this characterization in the proof of Theorem \ref{teor:complex:case}. 

We conclude the section with the definition of compatibility. 

\begin{deft}\label{def:compatibility}
Let $\kappa$ be a local field and let $\mathbf{H}$ be a $\kappa$-algebraic group. Consider $P$ a minimal parabolic subgroup of either $G_r$ or $G_c$. Let $U<P$ be the unipotent radical. We say that the pair $(\mathbf{H},\kappa)$ is \emph{compatible} with either $G_r$ or $G_c$ if for every non-trivial algebraic $\kappa$-subgroup $\mathbf{J}<\mathbf{H}$ and every continuous homormophism $\tau:P \rightarrow (N_{\mathbf{H}}(\mathbf{J})/\mathbf{J})(\kappa)$, we have that the Zariski closure of $\tau(U')$ coincides with the Zariski closure of $\tau(U)$, for every non-trival subgroup $U'<U$. 
\end{deft}

As noticed by Bader, Fisher, Miller, Stover \cite{BFMS21} compatibility is satisfied if the field is non-archimedean, because in this case $(N_{\mathbf{H}}(\mathbf{J})/\mathbf{J})(\kappa)$ is totally disconnected. Given a lattice $\Gamma<G_r$, if $\kappa$ is a local field extending the adjoint trace field $\ell$ of $\Gamma$ and $\mathbf{H}$ is the adjoint group of $\mathbf{G}(\kappa)$, where $\mathbf{G}$ is the $\ell$-group \cite{Vin71} containing $\Gamma$ (up to passing to a finite index subgroup), the pair $(\mathbf{H},\kappa)$ is compatible with $G_r$ \cite[Lemma 3.5]{BFMS21}. In the complex case, the same authors \cite[Proposition 3.4]{BFMS:Unpub} prove that $(\mathrm{PU}(r,s),\mathbb{R})$ is compatible with $G_c=\mathrm{SU}(n,1)$ when either $r+s<n+1$ or $s>1$. 

Here compatibility will be a crucial ingredient to show our trivializability statements. 

\section{Ergodic actions and algebraic varieties} \label{sec:ergodic:actions}

\subsection{Smooth actions} \label{sec:smoothness}

In this section we are going to recall the notion of smooth action. We will stress its crucial role in the study of actions of algebraic groups on either spaces of probability measures or spaces of rational functions. We mainly refer the reader to \cite[Chapter 2]{zimmer:libro}, \cite[Section 2.4]{SS21} for a more detailed discussion.  

We need first to recall the notion of countably separated space. A Borel space $(X,\mathcal{B})$ is \emph{countably separated} if there exists a countable family of Borel subsets $\{B_i \in \mathcal{B}\}_{i \in I}$ which separates points.

\begin{deft}\label{def:smooth:action}
Let $G$ be a locally compact group acting on a Borel space $(X,\mathcal{B})$. We say that the $G$-action is \emph{smooth} if the quotient Borel structure on $X/G$ is countably separated. 
\end{deft}

An example of smooth action is the one given by an algebraic group acting on algebraic variety. In our manuscript will be mainly interested in two different cases. The first one is an action of an algebraic group on the space of probability measures over a projective space. Before giving the details we are going to fix once and for all the following notation: if $X$ is a locally compact space, we denote by $\mathcal{M}^1(X)$ the space of all probability measures on it. Given a local field $\kappa$ and a finite dimensional vector $\kappa$-space $V$, we have a natural action of $\textup{PSL}(V)$ on the space $\mathcal{M}^1(\bbP(V))$ given by the pushforward construction, that is
$$
(g.\mu):=g_\ast \mu \ ,
$$
for every $g \in \textup{PSL}(V)$ and every $\mu \in \mathcal{M}^1(\bbP(V))$. 

\begin{prop}{\upshape \cite[Theorem 3.2.4, Theorem 3.2.6]{zimmer:libro}}\label{prop:stabilizer:measure}
Let $\kappa$ be a local field of characteristic zero and let $V$ be a $\kappa$-vector space of finite dimension. Then the $\textup{PSL}(V)$-action on the space $\mathcal{M}^1(\bbP(V))$ of probability measures is smooth. Additionally, if $\mu_0 \in \mathcal{M}^1(\bbP(V))$ is a probability measure, then the stabilizer $\textup{Stab}_{\textup{PSL}(V)}(\mu_0)$ is a compact extension of the $\kappa$-points of some $\kappa$-algebraic group.
\end{prop}

The second case that we are interested in is the smoothness of the joint action of two algebraic groups on the space of rational functions between algebraic varieties. We are not going to treat this subject in full generality, but we will focus our attention on a specific example. We denote by $G=\textup{SU}(n,1)$ and $H=\textup{PU}(n,1)$, for $n \geq 2$ (the assumption on $n$ is not necessary but it is exactly the same we have in our Theorem \ref{teor:complex:case}). It is well-known that both $G$ and $H$ can be viewed as the real points of their complexifications $\mathbf{G}=\textup{SL}(n+1,\bbC)$ and $\mathbf{H}=\textup{PSL}(n+1,\bbC)$, respectively. Additionally, there exist two parabolic subgroups $\mathbf{P}<\mathbf{G}$ and $\mathbf{Q}<\mathbf{H}$ such that $\partial_\infty \bbH^n_{\bbC}=(\mathbf{G}/\mathbf{P})(\bbR)=(\mathbf{H}/\mathbf{Q})(\bbR)$. We say that a selfmap of the boundary $\partial_\infty \bbH^n_{\bbC}$ is \emph{rational} if it is the restriction of a rational map between $\mathbf{G}/\mathbf{P}$ and $\mathbf{H}/\mathbf{Q}$. As a consequence we are allowed to speak about the set of rational functions $\Sigma:=\textup{Rat}(\partial_\infty \bbH^n_{\bbC},\partial_\infty \bbH^n_{\bbC})$. On the latter space we can define a joint $(G \times H)$-action given by 
$$
((g,h).\varphi)(\xi):=h\varphi(g^{-1}\xi) \ ,
$$
for every $g \in G, h \in H$ and $\varphi \in \Sigma$. We have the following

\begin{prop}{\upshape \cite[Proposition 3.3.2]{zimmer:libro}}\label{prop:smooth:action:rational}
Let $G=\textup{SU}(n,1)$ and $H=\textup{PU}(n,1)$. Then all the actions determined by $G$,$H$ and $G \times H$ on the space $\Sigma$ of rational functions of the boundary $\partial_\infty \bbH^n_{\bbC}$ are smooth actions. 
\end{prop}

We will see that smoothness will play a crucial role in the proofs of our theorems in order to exploit correctly the ergodicity of the space we are going to deal with. 

\subsection{Ergodicity and metric ergodicity} \label{sec:ergodicity}

In this section we are going to recall the main definitions and results about ergodic theory and its interaction with the world of algebraic varieties. We will start focusing our attention on a stronger notion of ergodicity, namely metric ergodicity. For more details about this section we mainly refer to \cite[Chapter 2]{zimmer:libro}, \cite{Bader:Furman:Unpub}, \cite[Section 2.4]{SS21}.

A \emph{standard Borel probability space} $(\Omega,\mu)$ is a probability space which is Borel isomorphic to a Polish space (that is a separable and completely metrizable space). If $\Gamma$ is a discrete countable group acting on $(\Omega,\mu)$ in a measure preserving way, we say that $(\Omega,\mu)$ is a standard Borel probability $\Gamma$-\emph{space}. If the action preserves only the measure class of $\mu$, we call $(\Omega,\mu)$ a \emph{Lebesgue} $\Gamma$-space. We say that the $\Gamma$-space is \emph{ergodic} if any measurable $\Gamma$-invariant (or quasi-invariant) subset of $\Omega$ has either full or null measure. Another important fact about ergodicity is the possibility to characterize it in terms of measurable invariant functions.  If $(\Omega,\mu)$ is an ergodic $\Gamma$-space, one can show \cite[Proposition 2.1.11]{zimmer:libro} that any Borel measurable $\Gamma$-invariant function with values in a countably separated space must be essentially constant. 

Here we want to introduce a stronger notion of ergodicity, allowing equivariant maps with metric spaces as targets. 

\begin{deft}\label{def:metric:ergodicity}
Let $\Gamma$ be a discrete countable group and let $(\Omega,\mu)$ be a Lebesgue $\Gamma$-space. Consider any separable metric space $(X,d)$ with an isometric $\Gamma$-action. We say that the $\Gamma$-action on $(\Omega,\mu)$ is \emph{metrically ergodic} if any $\Gamma$-equivariant measurable map $\Omega \rightarrow X$ is essentially constant.  
\end{deft}

Notice that metric ergodicity implies ergodicity by taking the discrete metric space $X=\{0,1\}$ with trivial $\Gamma$-action. As noticed by Bader and Furman \cite{Bader:Furman:Unpub}, metric ergodicity straightens the notion of ergodicity with \emph{unitary coefficients} introduced by Burger and Monod \cite{burger2:articolo}. In that case the allowed target can only be a Hilbert space with a unitary action, whereas here we allow any metric $\Gamma$-space. 

In what follows we will be particularly interested in the action of a lattice $\Gamma < G$ in an almost simple Lie group on the homogeneous quotients of the form $G/H$, where $H<G$ is a closed subgroup. As in the case of the Howe-Moore Theorem \cite[Theorem 2.2.20]{zimmer:libro} which gives us back an ergodicity criterion based on the non-compactness of $H$, here we have something analogous for metric ergodicity. 

\begin{prop}{\upshape \cite[Theorem 6.6, Corollary 6.7]{Bader:Gelander}}\label{prop:metrically:ergodic}
Let $G$ be a semi-simple Lie group with finite center and no compact factors. Let $\Gamma<G$ be a lattice. Suppose that $H$ is a closed subgroup such that $G/H$ has no precompact image in any proper quotient $G/G'$, where $G' \vartriangleleft G$ is a proper normal subgroup. Then the $\Gamma$-action on $G/H$ is metrically ergodic.
\end{prop}

Another important property of metric ergodicity is that it guarantees the ergodicity of products with respect to the diagonal action. More precisely, if $(\Omega,\mu)$ is a metrically ergodic $\Gamma$-space and $(\Theta,\nu)$ is an ergodic standard Borel probability $\Gamma$-space, their product $(\Omega \times \Theta,\mu \otimes \nu)$ is $\Gamma$-ergodic with respect to the diagonal action \cite[Corollary 6.9]{Bader:Gelander}.

\subsection{Representability of ergodic actions}\label{sec:representability:actions}

The notion of \emph{algebraic representability of an ergodic action} has been introduced by Bader and Furman \cite{Bader:Furman:Unpub,BF22} and it gives a criterion to translate a generic ergodic action on some standard Borel space in terms of some more natural algebraic action of the group on an algebraic variety.  

This section is devoted to the anologue of algebraic representability for \emph{measurable cocyles}. We are going to treat measurable cocycles in a more detailed way in Section \ref{sec:boundary:theory}. Here we simply recall their definition. Let $G,H$ be two locally compact second countable groups and let $\Gamma<G$ be a lattice. Consider a Lebesgue probability $\Gamma$-space $(\Omega,\mu)$. A Borel measurable function $\sigma:\Gamma \times \Omega \rightarrow H$ is a \emph{measurable cocycle} if 
$$
\sigma(\gamma_1 \gamma_2,s)=\sigma(\gamma_1,\gamma_2.s)\sigma(\gamma_2,s) \ ,
$$
for every $\gamma_1,\gamma_2 \in \Gamma$ and almost every $s \in \Omega$. 

\begin{deft}{\upshape \cite[Definition 9.1]{Bader:Furman:Unpub}}\label{def:algebraic:representability}
Let $\Gamma<G$ be a lattice in a second countable group and let $T<G$ be a closed subgroup. Consider a Lebesgue $(\Gamma \times T)$-space $(\Omega,\mu)$ and a $\kappa$-algebraic group $\mathbf{H}$, where $\kappa$ is a local field. A measurable cocycle $\sigma:\Gamma \times \Omega \rightarrow \mathbf{H}(\kappa)$ admits a $T$-\emph{algebraic representation} if there exist
\begin{itemize}
  \item a $\kappa$-algebraic group $\mathbf{L}$,
  \item a $\kappa$-algebraic $(\mathbf{H} \times \mathbf{L})$-variety $\mathbf{V}$,
  \item a continuous Zariski dense homomorphism $\tau:T \rightarrow \mathbf{L}(\kappa)$, 
  \item a measurable map $\Psi:\Omega \rightarrow \mathbf{V}(\kappa)$ which is $(\Gamma \times T)$-equivariant, that is
$$
\Psi(\gamma.s.t^{-1})=\sigma(\gamma,s)\Psi(s)\tau(t)^{-1} \ ,
$$
for every $\gamma \in \Gamma,t \in T$ and almost every $s \in \Omega$. 
\end{itemize}
We say that the tuple $(\mathbf{L},\mathbf{V},\tau,\Psi)$ is the $T$-\emph{algebraic representation} of $\sigma$. 
\end{deft}

In the same setting of the previous definition, suppose we have another $T$-algebraic representation $(\mathbf{L}',\mathbf{V}',\tau',\Psi')$. We define $\mathbf{M}$ as the Zariski closure of the product morphism $\tau \times \tau':T \rightarrow \mathbf{L} \times \mathbf{L}'$. Notice that $\mathbf{M}$ acts on both $\mathbf{V}$ and $\mathbf{V}'$ factoring via the projections on $\mathbf{L}$ and $\mathbf{L}'$. A \emph{morphism} of $T$-algebraic representations if a $\kappa$-algebraic map $\varphi:\mathbf{V} \rightarrow \mathbf{V}'$ which is $(\mathbf{H} \times \mathbf{M})$-equivariant and such that $\Psi'=\varphi \circ \Psi$. In this way we have built the \emph{category} of $T$-algebraic representations.

We say that a $T$-algebraic representation is a \emph{coset representations} if $\mathbf{V}=\mathbf{H}/\mathbf{J}$ for some $\kappa$-algebraic subgroup $\mathbf{J} < \mathbf{H}$ and $\mathbf{L}<N_\mathbf{H}(\mathbf{J})/\mathbf{J}$. For such a coset representation we write only the triple $(\mathbf{J},\tau,\Psi)$ to identify the $T$-algebraic representation. 

\begin{lem}\label{lem:algebraic:representation}
Let $G$ be a semi-simple Lie group, let $\Gamma<G$ be a lattice and let $T<G$ be a closed subgroup whose image is not precompact in any proper quotient of $G$. Consider an ergodic Lebesgue $\Gamma$-space $(\Omega,\mu)$ and a measurable cocycle $\overline{\sigma}:\Gamma \times \Omega \rightarrow \mathbf{H}(\kappa)$, where $\mathbf{H}$ is an algebraic group over a local field $\kappa$. Consider the $(\Gamma \times T)$-space $G \times \Omega$, where $\Gamma$ acts diagonally and $T$ acts only on the first factor. If $\sigma:\Gamma \times (G \times \Omega) \rightarrow \mathbf{H}(\kappa), \ \sigma(\gamma,(g,s))=\overline{\sigma}(\gamma,s),$ is the cocycle obtained by extending $\overline{\sigma}$, the category of $T$-algebraic representations of $\sigma$ has an initial object which is a coset representation of the form $(\mathbf{J},\tau,\Psi)$. 
\end{lem}

\begin{proof}
We follow the same strategy of \cite[Theorem 9.2]{Bader:Furman:Unpub}. We claim that it is sufficient to verify that any $T$-algebraic representation can be substituted with a coset representation, as in \cite[Lemma 9.3]{Bader:Furman:Unpub}. 

Let $(\mathbf{L},\mathbf{V},\tau,\Psi)$ be a $T$-algebraic representation of $\sigma$. We can consider the measurable cocycle
$$
\sigma \times \tau:(\Gamma \times T) \times (G \times \Omega) \rightarrow (\mathbf{H} \times \mathbf{L})(\kappa) \ , \ \ \ (\sigma \times \tau)((\gamma,t),(g,s)):=(\sigma(\gamma,(g,s)),\tau(t)) \ .
$$

By the non-compactness of $T$, we know by Proposition \ref{prop:metrically:ergodic} that the $\Gamma$-action on $T\backslash G$ is metrically ergodic, hence the $\Gamma$-action on $T \backslash G \times \Omega$ is ergodic by \cite[Corollary 6.9]{Bader:Gelander}. Thus the joint action of $\Gamma \times T$ on $G \times \Omega$ is ergodic. As a consequence of \cite[Proposition 5.2]{Bader:Furman:Unpub} applied to the cocycle $\sigma \times \tau$, we can suppose that $\mathbf{V}=(\mathbf{H} \times \mathbf{L})/\mathbf{M}$, where $\mathbf{M}<\mathbf{H} \times \mathbf{L}$ is an algebraic subgroup. 

We can consider the following commutative diagram
$$
\xymatrix{
G \times \Omega \ar[rr]^{\hspace{-10pt}\Psi} \ar[d] && ((\mathbf{H} \times \mathbf{L})/\mathbf{M})(\kappa) \ar[d]^{\pi_2} \\
G/\Gamma \times \Omega \ar[rr]^{\hspace{-10pt}\overline{\Psi}} &&  (\mathbf{L}/\pi_2(\mathbf{M}))(\kappa) \ .
}
$$
In the above diagram, $\pi_2$ is the projection on the second factor and we exploited the $\sigma$-equivariance of $\Psi$ to obtain the factorization of $\Psi$ through $\overline{\Psi}$. 

Since $T$ is not compact, the Howe-Moore Theorem \cite[Theorem 2.2.20]{zimmer:libro} guarantees that the $T$-action on $G/\Gamma$ is mixing, and hence weakly mixing. Since any slice of the map $\overline{\Psi}$ is a measurable $\tau$-equivariant map, by \cite[Proposition 3.3]{BF22} we must have a $\mathbf{L}$-fixed point in the target and hence $\mathbf{L}=\pi_2(\mathbf{M})$. Thus $\mathbf{V} \cong (\mathbf{H}/\pi_1(\mathbf{M}))$ is an algebraic representation of coset type. 

To conclude the proof it is sufficient to exploit the Noetherianity of $\mathbf{H}$ and to follow the line of the proof of \cite[Theorem 9.2]{Bader:Furman:Unpub}.   
\end{proof}
 
We stated the previous lemma in a slightly different way with respect to \cite[Theorem 9.2]{Bader:Furman:Unpub}. Our statement is actually more near to \cite[Theorem 4.3]{BF22} given for representations. The reason is that the main ingredient exploited to prove Lemma \ref{lem:algebraic:representation} is \cite[Proposition 3.3]{BF22}, which ensures that any measurable map obtained by representing algebraically the weakly mixing $T$-action on $G/\Gamma$ must be essentially constant and the image must be a fixed point. 

We conclude the section with the following result which give us a stability criterion of the measurable map appearing in the coset representation.

\begin{lem}{\upshape \cite[Lemma 4.4]{BFMS21}}\label{lem:same:map}
Let $\Gamma<G$ be a lattice and let $(\Omega,\mu)$ be an ergodic standard Borel probability space. Let $T<G$ be a closed non compact subgroup and let $\sigma:\Gamma \times \Omega \rightarrow \mathbf{H}(\kappa)$ be a Zariski dense measurable cocycle in the $\kappa$-points of some algebraic $\kappa$-group $\mathbf{H}$, where $\kappa$ is a local field. Suppose there exists a sequence $T=T_0 \vartriangleleft \cdots \vartriangleleft T_n=T'$ of subgroups of $G$ such that $T_i \vartriangleleft T_{i+1}$ for $i=0,\cdots,n-1$. If $(\mathbf{J},\tau,\Psi)$ and $(\mathbf{J}',\tau',\Psi')$ are the initial coset algebraic representations of $T$ and $T'$, respectively, then we can suppose $\mathbf{J}=\mathbf{J}'$ and $\phi=\phi'$ almost everywhere. 
\end{lem}

Thanks to an inductive argument, the key step in the proof is to show the above statement when $T < T' \leq N_G(T)$, where $N_G(T)$ is the $G$-normalizer of $T$. The same proof of \cite[Theorem 4.6]{Bader:Furman:Unpub} implies that the involved maps coincide almost everywhere. The fact that we deal with initial objects allows to prove that $\mathbf{J}=\mathbf{J}'$ and to get the desired statement. 

\section{Measurable cocycles and boundary theory}\label{sec:boundary:theory}

In the previous section we have seen the notion of measurable cocycle and how it can be algebraically represented. In this section we will give some more details about cocycles and about measurable maps that are equivariant with respect to them. For more details we refer the reader to \cite[Section 2.1]{SS21}. 

Let $G,H$ be two locally compact secount countable groups and let $\Gamma<G$ be a lattice. Consider a standard Borel probability $\Gamma$-space. Two measurable cocycles $\sigma_1,\sigma_2:\Gamma \times \Omega \rightarrow H$ are \emph{equivalent} (or \emph{cohomologous}) if there exists a measurable map $f:\Omega \rightarrow H$ such that
$$
\sigma_2(\gamma,s)=f(\gamma.s)^{-1}\sigma_1(\gamma,s)f(s) \ ,
$$
for every $\gamma \in \Gamma$ and almost every $s \in \Omega$. The words cocycle and cohomologous refer to the notion of cohomology of a countable equivalence relation on a probability space, introduced by Feldman and Moore \cite{feldman:moore}. In this case the countable equivalence relation is the orbital relation determined by the $\Gamma$-action on $(\Omega,\mu)$. 

An interesting source of measurable cocycles comes from representation theory. Indeed, given any representation $\rho:\Gamma \rightarrow H$, we can construct a measurable cocycle $\sigma_\rho:\Gamma \times \Omega \rightarrow H$ by setting $\sigma_\rho(\gamma,s):=\rho(\gamma)$. We say that a measurable cocycle $\sigma$ is \emph{trivializable} if $\sigma$ is cohomologous to a cocycle of the form $\sigma_\rho$.

Suppose now that $H$ corresponds to the $\kappa$-points of some $\kappa$-algebraic group $\mathbf{H}$, for some local field $\kappa$. It is well-known that the Zariski closure of the image of a representation $\Gamma \rightarrow \mathbf{H}(\kappa)$ is itself a group. In the context of measurable cocycles, even if the image does not have any nice algebraic structure, there exists a notion that reminds the one of Zariski closure for representations. The existence of such notion is guaranteed by the Noetherianity of the group $\mathbf{H}$. 

\begin{deft}\label{def:algebraic:group}
Let $\sigma:\Gamma \times \Omega \rightarrow \mathbf{H}(\kappa)$ be a measurable cocycle, with $(\Omega,\mu)$ an ergodic standard Borel probability space. The \emph{algebraic hull} of $\sigma$ is the (conjugacy class of the) smallest $\kappa$-algebraic subgroup $\mathbf{L}<\mathbf{H}$ which contains the image of a cocycle cohomologous to $\sigma$. We say that $\sigma$ is \emph{Zariski dense} if $\mathbf{L}=\mathbf{H}$. 
\end{deft}

An important tool in the study of measurable cocycles is given by measurable equivariant maps. More precisely we have the following

\begin{deft}\label{def:equivariant:map:cocycle}
Let $\Gamma < G$ be a lattice in a simple Lie group and let $(\Omega,\mu)$ be a standard Borel probability $\Gamma$-space. Consider a $\kappa$-algebraic group $\mathbf{H}$ and some $\kappa$-algebraic $\mathbf{H}$-variety $\mathbf{V}$, where $\kappa$ is a local field. Let $\sigma:\Gamma \times \Omega \rightarrow \mathbf{H}(\kappa)$ be a measurable cocycle. Fix $W<G$ a closed subgroup. A measurable map $\Phi:W \backslash G \times \Omega \rightarrow \mathbf{V}(\kappa)$ is $\Gamma$-\emph{equivariant} (or $\sigma$-\emph{equivariant}) if 
$$
\Phi(\gamma.gW,\gamma.s)=\sigma(\gamma,s)\Phi(gW,s) \ ,
$$
for every $\gamma \in \Gamma$ and almost every $g \in G, s \in \Omega$. 
\end{deft}

Given a measurable equivariant map $\Phi:W\backslash G \times \Omega \rightarrow \mathbf{V}(\kappa)$ as in Definition \ref{def:equivariant:map:cocycle}, we know that the map 
$$
\Phi_s:W\backslash G \rightarrow \mathbf{V}(\kappa)  \ , \ \ \Phi_s(gW):=\Phi(gW,s) \ 
$$
is measurable for almost every $s \in\ \Omega$ \cite[Chapter VII, Lemma 1.3]{margulis:libro}. We call $\Phi_s$ the $s$-\emph{slice} associated to $\Phi$. By the $\sigma$-equivariance of the map $\Phi$, we have that
\begin{equation}\label{eq:equivariance:slices}
\Phi_{\gamma.s}(\ \cdot \ )=\sigma(\gamma,s)\Phi_s(\gamma^{-1} \cdot \ ) \ . 
\end{equation}

A particular example of equivariant maps is given by boundary maps. We consider $G=\textup{SU}(n,1)$, $\Gamma < G$ a lattice, $(\Omega,\mu)$ a standard Borel probability $\Gamma$-space and $W=P<G$ a minimal parabolic subgroup. Let $H=\textup{PU}(n,1)$ and $\mathbf{V}(\bbR)=\partial_\infty \bbH^n_{\bbC}$ (by Section \ref{sec:smoothness} we know that $\partial_\infty \bbH^n_{\bbC}$ corresponds to the real points of an algebraic variety). A \emph{boundary map} for a measurable cocycle $\sigma:\Gamma \times \Omega \rightarrow H$ is simply a measurable equivariant map $\Phi: \partial_\infty \bbH^n_{\bbC} \times \Omega \rightarrow \partial_\infty \bbH^n_{\bbC}$. 

We conclude the section with the following proposition which characterizes the slices of a boundary map associated to a Zariski dense cocycle.

\begin{prop}{\upshape \cite[Proposition 4.4]{SS21}}\label{prop:slice:zariski:dense}
Let $\Gamma < \textup{SU}(n,1)$ be a lattice, where $n \geq 2$. Let $(\Omega,\mu)$ be an ergodic standard Borel probability $\Gamma$-space. Consider a Zariski dense cocycle $\sigma:\Gamma \times \Omega \rightarrow \textup{PU}(n,1)$. If $\phi:\partial_\infty \bbH^n_{\bbC} \times \Omega \rightarrow \partial_\infty \bbH^n_{\bbC}$ is a boundary map for $\sigma$, then the slice $\phi_s$ has Zariski dense essential image for almost every $s \in \Omega$. 
\end{prop}

\section{Superrigidity for real cocycles}\label{sec:real:case}

In this section we are going to prove our superrigidity theorem in the case of real hyperbolic lattices. The first step will be the construction of a measurable equivariant map whose target is a quasi-projective variety constructed using the compatibility datum. Exploiting the algebraic representability of ergodic actions described in Section \ref{sec:representability:actions} we will construct two morphisms defined on two different subgroups that generate $G$. The possibility to glue those morphisms together will lead to the desired statement. 

We start by fixing the following

\begin{setup}\label{setup:real:case}
 Consider $n \geq 3$ and $G=\textup{SO}^\circ(n,1)$. We assume that:
	\begin{itemize}
	\item $\Gamma \leq G$ is a lattice and $W<G$ is a simple non-compact Lie group;
	\item $(\Omega,\mu)$ is an ergodic standard Borel probability $\Gamma$-space;
	\item $\kappa$ is a local field, $\mathbf{H}$ is a connected simple $\kappa$-algebraic group with $\kappa$-points $H=\mathbf{H}(\kappa)$ and such that $(\mathbf{H},\kappa)$ is compatible with $G$;
	\item $\sigma:\Gamma \times \Omega \rightarrow H$ is a Zariski dense measurable cocycle with unbounded image;
	\item there exists an irreducible $\kappa$-representation of $\mathbf{H}$ on a $\kappa$-vector space $V$;
	\item $\Phi:G \times \Omega \rightarrow \mathcal{M}^1(\bbP(V) )$ is a measurable map such that $\Phi(\gamma.g,\gamma.s)=\sigma(\gamma,s)_\ast \Phi(g,s)$ for every $\gamma \in \Gamma$ and almost every $g \in G, s \in \Omega$, and $\Phi$ is $W$-invariant on the first variable.
	\end{itemize}
\end{setup}

\begin{teor} \label{teor:equivariant:map:real:case}
Under the assumptions of Setup \ref{setup:real:case}, there exists an algebraic $\kappa$-subgroup $\mathbf{L}$ of $\mathbf{H}$ and a measurable $\sigma$-equivariant map $\Phi: G \times \Omega \rightarrow (\mathbf{H}/\mathbf{L})(\kappa)$.
Additionally $\Phi$ is $W$-invariant with respect to the first factor.  
\end{teor}

\begin{proof}
Since the map $\Phi$ is $W$-invariant with respect to the first variable, we can define
$$
\overline{\Phi}:W \backslash G \times \Omega \rightarrow \mathcal{M}^1(\bbP(V)) \ .
$$

By Proposition \ref{prop:stabilizer:measure}, we know that $H$ acts on $\mathcal{M}^1(\bbP(V))$ in a smooth way, that is the quotient $\Sigma:=\mathcal{M}^1(\bbP(V))/H$ is countably separated. We can consider the composition of the map $\overline{\Phi}$ with the projection on the quotient $\Sigma$ to obtain the following map
$$
\widehat{\Phi}:W \backslash G \times \Omega \rightarrow \Sigma \ .
$$
The $\sigma$-equivariance of $\Phi$ implies that the map $\widehat{\Phi}$ is $\Gamma$-invariant. Since $W$ is not compact, Proposition \ref{prop:metrically:ergodic} implies that the $\Gamma$ action on $W \backslash G$ is metrically ergodic. The ergodicity of the probability space $(\Omega,\mu)$ together with \cite[Corollary 6.9]{Bader:Gelander} imply that the $\Gamma$-action on $W \backslash G \times \Omega$ is metrically ergodic, and hence ergodic. As a consequence the map $\widehat{\Phi}$ is essentially constant. Equivalently the map $\overline{\Phi}$ takes values in a unique $H$-orbit. This means that there exists some measure $\mu_0 \in \mathcal{M}^1(\bbP(V))$ such that
$$
\overline{\Phi}: W \backslash G \times \Omega \rightarrow H.\mu_0 \cong H/\textup{Stab}_H(\mu_0) \ ,
$$
where $\textup{Stab}_H(\mu_0)$ is the stabilizer of the measure $\mu_0$ and the isomorphism holds thanks to the smoothness of the $H$-action on $\mathcal{M}^1(\bbP(V))$. We denote by $L=\textup{Stab}_H(\mu_0)$. By Proposition \ref{prop:stabilizer:measure} we know that the stabilizer $L$ is a compact extension of the $\kappa$-points of some algebraic $\kappa$-group. 

We want to show that $L$ is not compact. Suppose by contradiction that $L$ is compact. As a consequence there exists a $H$-invariant metric $d_H$ on $H/L$. We know that the $\Gamma$-action on $W\backslash G$ is metrically ergodic. We consider the space $\upL^0(\Omega, H /L)$ of (equivalence classes) of measurable functions with the distance given by 
$$
d_\mu(f,g):=\int_\Omega \frac{d_H(f(s),g(s))}{1+d_H(f(s),g(s))} d\mu(s) \ .
$$ 
Notice that such distance is well-defined since the argument in the integral above is bounded and $\mu$ is a probability measure. We have a natural $\Gamma$-action on the separable space $\upL^0(\Omega,H/L)$ determined by the cocycle $\sigma$. More precisely, for every $\gamma \in \Gamma$ and every (equivalence class of) measurable function $f$, we define
\begin{equation}\label{eq:gamma:action:cocycle}
(\gamma . f)(s):=\sigma(\gamma,s)f(\gamma^{-1}.s) \ ,
\end{equation}
for almost every $s \in \Omega$. The function $\overline{\Phi}$ leads to a measurable function 
$$
\widetilde{\Phi}: W \backslash G \rightarrow \upL^0(\Omega,H/L) \ , \ \ (\widetilde{\Phi}(gW))(s):=\overline{\Phi}(gW,s) \ . 
$$
Additionally the $\sigma$-equivariance of the function $\overline{\Phi}$ implies that the function $\widetilde{\Phi}$ is $\Gamma$-equivariant with respect to the action defined by Equation \eqref{eq:gamma:action:cocycle}. By the metric ergodicity of $W \backslash G$ this means that $\widetilde{\Phi}$ is essentially constant. Equivalently the map $\overline{\Phi}$ does not depend on the first variable. In this way we obtain a measurable map $\Omega \rightarrow H/L$ defined by considering the essential image of the slice $\Phi_s(gW)=\Phi(gW,s)$.  We obtained a measurable map which is $\sigma$-equivariant, violating the Zariski density of the cocycle $\sigma$ (otherwise the cocycle would essentially take values in $L$, see \cite[Proposition 3.3]{Bader:Furman:Unpub}). This proves that $L$ cannot be compact.

If we now consider $\mathbf{L}=\overline{L}^Z$ the Zariski closure of $L$ in $\mathbf{H}$, we obtain a $\kappa$-subgroup of $\mathbf{H}$ \cite[Proposition 3.2.15]{zimmer:libro}. Up to composing $\Phi$ with the projection $H/L \rightarrow (\mathbf{H}/\mathbf{L})(\kappa)$, we obtain the desired map and the statement is proved. 
\end{proof}

The equivariant map obtained in Theorem \ref{teor:equivariant:map:real:case}, together with the algebraic representability of the ergodic action of $\Gamma \times T$ on $G \times \Omega$ will be the main ingredients to prove our main theorem in the real case. 

\begin{proof}[Proof of Theorem \ref{teor:real:case}]
Let $P<G$ be a minimal parabolic subgroup (say the one stabilizing the point at infinity in the upper-half space model) and consider $U<P$ its unipotent radical. We define 
$$
U'=U \cap W \ .
$$
Clearly the group $U'$ is not compact. By Lemma \ref{lem:algebraic:representation}, we have a $U'$-algebraic representation of the space $G \times \Omega$, that is there exist
\begin{itemize}
	\item a $\kappa$-algebraic subgroup $\mathbf{J} < \mathbf{H}$, 
	\item a continuous Zariski dense homomorphism $\tau:U' \rightarrow (N_\mathbf{H}(\mathbf{J}) /\mathbf{J})(\kappa)$,
	\item a measurable map $\Psi: G \times \Omega \rightarrow (\mathbf{H}/\mathbf{J})(\kappa)$ such that 
$$
\Psi(\gamma.g.u^{-1}, \gamma.s)=\sigma(\gamma,s)\Psi(g,s)\tau(u)^{-1} \ , 
$$
for every $\gamma \in \Gamma, u \in U'$ and almost every $g \in G, s \in \Omega$. 
\end{itemize}

Notice that by Lemma \ref{lem:same:map}, we know that $U'$ and $P$ admit the same map as algebraic representations (since $U' \vartriangleleft P$).  More precisely we can extend $\tau$ to a continuous homomorphism defined on $P$ mantaining the same map $\Psi$. 

We want to show that the subgroup $\mathbf{J}$ must be trivial. Suppose by contradiction that $\mathbf{J}$ is not trivial. We consider the measurable map $\Phi:G \times \Omega \rightarrow (\mathbf{H}/\mathbf{L})(\kappa)$ defined in Theorem \ref{teor:equivariant:map:real:case}. The fact that $\Psi$ is the initial object in the category of algebraic representations implies that we must have the following commutative diagram
\begin{equation}\label{eq:diagram:equivariant:map}
\xymatrix{
G \times \Omega \ar[rr]^{\Phi} \ar[d]^{\Psi} && (\mathbf{H}/\mathbf{L})(\kappa)  \\
(\mathbf{H}/\mathbf{J})(\kappa) \ar@{..>}[rru] \ar[rr] && (\mathbf{H}/\mathbf{J})(\kappa)/(\overline{\tau(U')})=(\mathbf{H}/\mathbf{J})(\kappa)/(\overline{\tau(U)}) \ar@{..>}[u]\ .
}
\end{equation}

Since $\Phi$ is $W$-invariant on the first factor, it is a \emph{fortiori} $U'$-invariant because $U'$ is a subgroup of $W$. Such invariance, together with the compatibility assumption on the pair $(\mathbf{H},\kappa)$, implies that $\Phi$ factors actually through $(\mathbf{H}/\mathbf{J})(\kappa)/(\overline{\tau(U)})$. This means that $\Phi$ is both $W$-invariant and $U$-invariant, thus it is $\langle W, U \rangle=G$-invariant. Equivalently, the measurable map $\Phi_s:G \rightarrow (\mathbf{H}/\mathbf{L})(\kappa), 	\ \Phi_s(g):=\Phi(g,s)$ is essentially constant for almost every $s \in \Omega$. In this way we obtain a measurable $\sigma$-equivariant map 
$$
\Omega \rightarrow (\mathbf{H}/\mathbf{L})(\kappa) \ ,
$$
and by \cite[Proposition 3.3]{Bader:Furman:Unpub} the cocycle $\sigma$ cannot be Zariski dense, leading to a contradiction. 

Since $\mathbf{J}$ must be trivial, we obtain a continuous Zariski dense homomorphism $\tau:P \rightarrow H$ and a measurable $(\Gamma \times P)$-equivariant map $\Psi:G \times \Omega \rightarrow H$, that is
$$
\Psi(\gamma.(g,s).p^{-1})=\sigma(\gamma,s)\Psi(g,s)\tau(p)^{-1} \ ,
$$
for every $\gamma \in \Gamma, p \in P$ and almost every $g \in G, s \in \Omega$. If we restrict $\tau$ to the maximal $\bbR$-split torus $A$ in $P$, we have that $\Psi$ is the map of the initial object in the category of $A$-algebraic representations of $G \times \Omega$. By Lemma \ref{lem:same:map} we can extend $\tau|_A$ to a continuous Zariski dense morphism $\tau':N_G(A) \rightarrow H$ in such a way that $\Psi$ is $(\Gamma \times N_G(A))$-equivariant. Here $N_G(A)$ is the normalizer of $A$ in $G$.

We now set $T_1:=P \ , \  T_2:=N_G(A)$ and similarly $\tau_1:=\tau\ ,\ \tau_2:=\tau'$. We have that $\Phi$ is $(\Gamma \times T_i)$-equivariant, that is
\begin{equation}\label{eq:phi:equivariance}
\Psi(\gamma.(g,s).t_i^{-1})=\sigma(\gamma,s)\Psi(g,s)\tau_i(t_i)^{-1} \ ,
\end{equation}
for every $\gamma \in \Gamma,t_i \in T_i$ for $i=1,2$ and almost every $g \in G, s \in \Omega$. Since the groups $T_1$ and $T_2$ generate $G$, that is $\langle T_1,T_2 \rangle =G$, we are in the right position to apply \cite[Lemma 5.1]{BF22} to the space $G$ with the right $G$-translation and to the measurable $s$-slice
$$
\Psi_s:G \rightarrow H \ , \ \ \Psi_s(g):=\Psi(g,s) \ ,
$$
for almost every $s \in \Omega$. By \cite[Lemma 5.1]{BF22} there must exist a continuous morphism 
$$
\Upsilon_s:G \rightarrow H \ , 
$$
such that $\Psi_s$ is $G$-equivariant in the following sense
\begin{equation}\label{eq:equivariance:slice:upsilon}
\Psi_s(gh^{-1})=\Psi_s(g)\Upsilon_s(h)^{-1} \ ,
\end{equation}
for every $h \in G$. A priori $\Upsilon_s$ may depend on $s \in \Omega$, but by \cite[Lemma 5.1]{BF22} the restriction of $\Upsilon_s$ to each subgroup $T_i$ must coincide with $\tau_i$, for $i=1,2$. Since $T_1,T_2$ generate $G$, the morphism $\Upsilon_s$ is uniquely determined and it does not depend on the $s$-variable. We are going to denote it by $\Upsilon:G \rightarrow H$. 

To conclude we follow the line of \cite[Theorem 1.3]{BF22}. We report here the details for sake of completness. We define
$$
\varphi:\Omega \rightarrow H \ , \ \ \ \varphi(s):=\Psi(g,s)\Upsilon(g)^{-1} \ .
$$
By the $G$-equivariance of $\Psi$, due to Equation \eqref{eq:equivariance:slice:upsilon}, the map above does not depend on $g \in G$. We can now look at 
\begin{equation}\label{eq:first:phi}
\Psi(\gamma.(g,s))=\sigma(\gamma,s)\Psi(g,s)=\sigma(\gamma,s)\varphi(s)\Upsilon(g) \ . 
\end{equation}
Similarly it holds
\begin{equation}\label{eq:second:phi}
\Psi(\gamma.(g,s))=\varphi(\gamma.s)\Upsilon(\gamma)\Upsilon(g) \ .
\end{equation}
By comparing Equation \eqref{eq:first:phi} and Equation \eqref{eq:second:phi} we obtain that
$$
\Upsilon(\gamma)=\varphi(\gamma.s)^{-1}\sigma(\gamma,s)\varphi(s) \ ,
$$
and the statement is proved. 
\end{proof}

In the next section we will see that the argument given above is not sufficient in the complex case. In fact it may happen that the pair $(\mathbf{H},\kappa)$ is not compatible and this corresponds exactly to the case when the subgroup $\mathbf{J}$ of the $P$-algebraic representation of $G \times \Omega$ is not trivial (one can think to $\mathbf{J}$ as a subgroup measuring the compatibility of the pair $(\mathbf{H},\kappa)$). 

\section{Superrigidity for complex cocycles}\label{sec:complex:case}

In this section we will focus our attention on the superrigidity statement in the complex case. We will see that the proof is identical to the one given in Section \ref{sec:real:case} if the pair $(\mathbf{H},\kappa)$ is compatible with the group $\textup{SU}(n,1)$. If the pair is not compatible, then the algebraic representability leads to a boundary map in the sense of \cite[Definition 2.9]{moraschini:savini:2}. Since the slices of such boundary map are Zariski dense and chain preserving, we can apply the same strategy of the proof of \cite[Theorem 2]{SS21} to get the desired trivialization for the starting cocycle.

We start by fixing the following

\begin{setup}\label{setup:complex:case}
 Consider $n \geq 2$ and $G=\textup{SU}(n,1)$. We assume that:
	\begin{itemize}
	\item $\Gamma \leq G$ is a lattice and $W<G$ is a simple non-compact Lie group;
	\item $(\Omega,\mu)$ is an ergodic standard Borel probability $\Gamma$-space;
	\item $\kappa$ is a local field, $\mathbf{H}$ is a connected simple $\kappa$-algebraic group with $\kappa$-points $H=\mathbf{H}(\kappa)$;
	\item $\sigma:\Gamma \times \Omega \rightarrow H$ is a Zariski dense measurable cocycle with unbounded image;
	\item there exists an irreducible $\kappa$-representation of $\mathbf{H}$ on a $\kappa$-vector space $V$;
	\item $\Phi:G \times \Omega \rightarrow \mathcal{M}^1(\bbP(V) )$ is a measurable map such that $\Phi(\gamma.g,\gamma.s)=\sigma(\gamma,s)_\ast \Phi(g,s)$ and $\Phi(g,s)$ for every $\gamma \in \Gamma$ and almost every $g \in G, s \in \Omega$, and $\Phi$ is $W$-invariant with respect to the first variable. 
	\end{itemize}
\end{setup}

\begin{teor} \label{teor:equivariant:map:complex:case}
Under the assumptions of Setup \ref{setup:complex:case}, there exists an algebraic $\kappa$-subgroup $\mathbf{L}$ of $\mathbf{H}$ and a measurable $\sigma$-equivariant map $\Phi: G \times \Omega \rightarrow (\mathbf{H}/\mathbf{L})(\kappa)$. Additionally $\Phi$ is $W$-invariant with respect to the first factor.  
\end{teor}

\begin{proof}
It is sufficient to adapt \emph{mutatis mutandis} the proof of Theorem \ref{teor:equivariant:map:real:case} to this context. 
\end{proof}

Now we have an equivariant map given by Theorem \ref{teor:equivariant:map:complex:case}. We want to use such a map together with the theory of algebraic representability of ergodic actions. This time we will face a crucial difference with respect to the proof of Theorem \ref{teor:real:case}. The group $\mathbf{J}$ appearing in that proof can be not trivial in this setting. The non-triviality of $\mathbf{J}$ will exactly happen when the pair $(\mathbf{H},\kappa)$ is not compatible with $G$. 

\begin{prop}\label{prop:j:non:trivial}
Under the assumptions of Setup \ref{setup:complex:case}, let $U'<W$ be a non-trivial unipotent subgroup. Consider a $U'$-algebraic representation of $G \times \Omega$ given by the triple $(\mathbf{J},\tau,\Psi)$.  Then $\tau$ can be extended to continuous Zariski dense morphism $\tau':P \rightarrow (N_\mathbf{H}(\mathbf{J})/\mathbf{J})(\kappa)$ such that $\Psi$ is a measurable $(\Gamma \times P)$-equivariant map. Additionally, if $\mathbf{J}$ is not trivial then the pair $(\mathbf{H},\kappa)$ is not compatible. 
\end{prop}

\begin{proof}
Except for the non-compatibility statement, this is exactly the first part of the proof of Theorem \ref{teor:real:case}. We quickly review it. As a consequence of Lemma \ref{lem:algebraic:representation} there must exist
\begin{itemize}
  \item a $\kappa$-algebraic subgroup $\mathbf{J}<\mathbf{H}$,
  \item a continuous Zariski dense morphism $\tau:U' \rightarrow (N_\mathbf{H}(\mathbf{J})/\mathbf{J})(\kappa)$,
  \item a measurable map $\Psi:G \times \Omega \rightarrow (\mathbf{H}/\mathbf{J})(\kappa)$ which is $(\Gamma \times U')$-equivariant, that is
$$
\Psi(\gamma.(g,s).u^{-1})=\sigma(\gamma,s)\Psi(g,s)\tau(u)^{-1} \ ,
$$
for every $\gamma \in \Gamma, u \in U'$ and almost every $g \in G,s \in \Omega$. 
\end{itemize}

By Lemma \ref{lem:same:map}, we know that $U'$ and $P$ admit the same map as algebraic representations (since $U' \vartriangleleft P$).  This precisely means that we can extend $\tau$ to a continuous homomorphism defined on $P$ mantaining the same map $\Psi$. 

We are left to show that if the subgroup $\mathbf{J}$ is not trivial, then the pair $(\mathbf{H},\kappa)$ cannot be compatible. By contradiction suppose that $(\mathbf{H},\kappa)$ is compatible. Theorem \ref{teor:equivariant:map:complex:case} guarantees the existence of measurable $\sigma$-equivariant $W$-invariant map $\Phi:G \times \Omega \rightarrow (\mathbf{H}/\mathbf{L})(k)$, for some $\kappa$-algebraic subgroup $\mathbf{L}< \mathbf{H}$. Additionally $\Phi$ fits in a commutative diagram which is analogous to Diagram \eqref{eq:diagram:equivariant:map}. Like in the proof of Theorem \ref{teor:real:case}, we argue that $\Psi$ must be $G$-invariant and thus $\sigma$ cannot be Zariski dense, getting a contradiction. This proves the statement and concludes the proof. 
\end{proof}

Before giving the full proof of Theorem \ref{teor:complex:case} we need to handle the case when $\mathbf{J}$ is not trivial, or equivalently when the pair $(\mathbf{H},\kappa)$ is not compatible. 

\begin{prop}\label{prop:rigidity:boundary:map}
Let $G=\textup{SU}(n,1)$ for $n \geq 2$ and let $\Gamma < G$ be a lattice. Let $H=\textup{PU}(n,1)$, $Q<H$ be a minimal parabolic subgroup and let $\overline{Z}<Q$ be its center. Consider an ergodic standard Borel probability $\Gamma$-space $(\Omega,\mu)$ and a Zariski dense cocycle $\sigma:\Gamma \times \Omega \rightarrow H$ with unbounded image. Suppose that there exist
\begin{itemize}
	\item a continuous homomorphism $\tau:P \rightarrow Q/\overline{Z}$ with one-dimensional kernel,
	\item a measurable map $\phi:G \times \Omega \rightarrow H/\overline{Z}$ which is $(\Gamma \times P)$-equivariant, that is
$$
\phi(\gamma.(g,s).p^{-1})=\sigma(\gamma,s)\phi(g,s)\tau(p)^{-1} \ ,
$$
for every $\gamma \in \Gamma, p \in P$ and almost every $g \in G,s \in \Omega$. 
\end{itemize}
Then $\sigma$ is trivializable. 
\end{prop}

\begin{proof}
We consider the composition of the map $\phi$ with the projection $H/\overline{Z} \rightarrow H/Q$. We denote by $\overline{\phi}$ such composition. Notice that the fact that $\phi$ is $P$-equivariant, implies that $\overline{\phi}$ is $P$-invariant (because the image of $\tau$ lies in $Q/ \overline{Z}$). As a consequence $\overline{\phi}$ induces another measurable map  
$$
\widehat{\phi}:P \backslash G \times \Omega \rightarrow H/Q \ . 
$$
Notice that both the quotients $P \backslash G$ and $H/Q$ can be identified with the boundary at infinity $\partial_\infty \bbH^n_{\bbC}$. Moreover, the $\Gamma$-equivariance of $\phi$ implies that 
$$
\widehat{\phi}(\gamma.\xi,\gamma.s)=\sigma(\gamma,s)\widehat{\phi}(\xi,s) \ ,
$$
for every $\gamma \in \Gamma$ and almost every $\xi \in \partial_\infty \bbH^n_{\bbC}, s \in \Omega$. This means exactly that $\widehat{\phi}$ is a boundary map in the sense of \cite[Definition 2.9]{moraschini:savini:2}. If we now show that the slice
$$
\widehat{\phi}_s:\partial_\infty \bbH^n_{\bbC} \rightarrow \partial_\infty \bbH^n_{\bbC} \ , \ \ \widehat{\phi}_s(\xi):=\widehat{\phi}(\xi,s) \ ,
$$ 
has Zariski dense essential image and it is chain preserving, then we can apply the same strategy of \cite[Theorem 2]{SS21} to conclude.  

Since $\sigma$ is Zariski dense and $(\Omega,\mu)$ is ergodic, the Zariski density of the essential image of $\widehat{\phi}_s$ follows directy by Proposition \ref{prop:slice:zariski:dense}. 

We are left to show that $\widehat{\phi}_s$ is chain preserving for almost every $s \in \Omega$. We first show that it is essentially injective. Consider the measurable set
$$
E:=\{ (\xi,\eta,s) \in (\partial_\infty \bbH^n_{\bbC})^2 \times \Omega \ | \ \widehat{\phi}_s(\xi)=\widehat{\phi}_s(\eta) \} \ .
$$
By the $\sigma$-equivariance of $\widehat{\phi}$ we have that the set $E$ is a $\Gamma$-invariant measurable subset of $(\partial_\infty \bbH^n_{\bbC})^2 \times \Omega$. Since the diagonal action of $\Gamma$ on $(\partial_\infty \bbH^n_{\bbC})^2 \times \Omega$ is ergodic \cite[Proposition 3.3]{MonShal0}, we have that $E$ has either full or null measure. We claim that $E$ must have null measure. By contradiction suppose that $E$ has full measure. In that case we must have
$$
\widehat{\phi}_s(\xi)=\widehat{\phi}_s(\eta) \ ,
$$
for almost every $\xi,\eta \in \partial_\infty \bbH^n_{\bbC},s \in \Omega$. Thus $\widehat{\phi}_s$ is essentially constant for almost every $s \in \Omega$. In this way we obtain a measurable map $\Omega \rightarrow \partial_\infty \bbH^n_{\bbC}$ which is $\sigma$-equivariant. By \cite[Proposition 3.3]{Bader:Furman:Unpub} the cocycle $\sigma$ cannot be Zariski dense. This is a contradiction and we must have that $E$ is of null measure. As a consequence 
$
\widehat{\phi}_s(\xi) \neq \widehat{\phi}_s(\eta) \ ,
$
for almost every $\xi,\eta \in \partial_\infty \bbH^n_{\bbC},s \in \Omega$, or equivalently $\widehat{\phi}_s$ is essentially injective for almost every $s \in \Omega$. The essential injectivity implies that $\widehat{\phi}_s$ restricts to a well-defined map 
$$
\widehat{\phi}_s^{(2)}: (\partial_\infty \bbH^n_{\bbC})^{(2)} \rightarrow (\partial_\infty \bbH^n_{\bbC})^{(2)} \ ,
$$
on the subset $(\partial_\infty \bbH^n_{\bbC})^{(2)}$ of distinct points in $(\partial_\infty \bbH^n_{\bbC})^{2}$.  

Now we prove that $\widehat{\phi}_s$ is chain preserving. We follow the line of the proof of \cite[Proposition 6.1]{BFMS:Unpub}. Following the notation of Section \ref{sec:hyperbolic:groups}, recall that $C_{\xi,\eta}$ is the unique chain passing through two distinct points $\xi,\eta \in \partial_\infty \bbH^n_{\bbC}$ and $\mathcal{C}_1$ is the space parametrizing the pairs $(C,\eta)$ where $C$ is a chain and $\eta \in C$. We also have the natural projection
$$
\pi:(\partial_\infty \bbH^n_{\bbC})^{(2)} \rightarrow \mathcal{C}_1 \ , \ \ \pi(\xi,\eta):=(C_{\xi,\eta},\xi) \ .
$$
We can identify $(\partial_\infty \bbH^n_{\bbC})^{(2)}$ with the quotient $MA\backslash G$ and similarly $\mathcal{C}_1$ with the quotient $(H/\overline{Z})/\theta(P \cap D)$, where the latter is a consequence of the fact $\theta(P \cap D)=\tau(P \cap D)$ (see Section \ref{sec:compatibility}). 

The same reasoning used to prove \cite[Proposition 6.1]{BFMS:Unpub} shows the existence of the dashed arrows $\psi_s$ and $\widehat{\psi}_s$ in the following commutative diagram
\begin{equation}\label{eq:diagram:dashed:arrows}
\xymatrix{
G \ar[d] &  &&   & \\
MA \backslash G \ar[r]^{\cong} \ar[d] & (\partial_\infty \bbH^n_{\bbC})^{(2)} \ar[rr]^{\widehat{\phi}_s^{(2)}}  \ar[d]^{\pi} && (\partial_\infty \bbH^n_{\bbC})^{(2)} \ar[d]^{\pi} & \\
(P\cap D) \backslash G \ar[r]^{\cong} \ar@{..>}@/_2pc/[rrrr]_{\psi_s} & \mathcal{C}_1 \ar@{..>}[rr]^{\widehat{\psi}_s} && \mathcal{C}_1 \ar[r]^{\cong} & (H/\overline{Z})/\theta(P \cap D) \ .
}
\end{equation}

More precisely, the $\tau$-equivariance of the map $\phi$ (and of the associated slice $\phi_s$) and the fact that $\theta(P \cap D)=\tau(P \cap D)$ (see Section \ref{sec:compatibility}), imples that the composition $G \rightarrow (H/\overline{Z})\slash \theta(P \cap D)$ in the Diagram \ref{eq:diagram:dashed:arrows} factors through $(P \cap D) \backslash G$. This leads to the definition of $\psi_s$ and composing with the other two isomorphisms we get the map $\widehat{\psi}_s$. In particular we obtain a map $\widehat{\psi}_s:\mathcal{C}_1 \rightarrow \mathcal{C}_1$ such that the following diagram is commutative:
$$
\xymatrix{
(\partial_\infty \bbH^n_{\bbC})^{(2)} \ar[rr]^{\widehat{\phi}_s^{(2)}}  \ar[d]^{\pi} && (\partial_\infty \bbH^n_{\bbC})^{(2)} \ar[d]^{\pi} \\
\mathcal{C}_1 \ar@{..>}[rr]^{\widehat{\psi}_s} && \mathcal{C}_1 \ .
}
$$

This exactly means that the slice $\widehat{\phi}_s$ is chain preserving. Being Zariski dense and chain preserving, we can apply \cite[Theorem 1.6]{Pozzetti} to deduce that $\widehat{\phi}_s$ is rational (in the sense given in Section \ref{sec:smoothness}) for almost every $s \in \Omega$. 

Now we can conclude using the same argument contained in the proof of \cite[Theorem 2]{SS21}. We report the main ideas of the proof for sake of completeness, but we suggest the reader to look at \cite[Theorem 2]{SS21} for more details. Consider the map 
$$
\Phi:\Omega \rightarrow \Sigma:=\textup{Rat}(\partial_\infty \bbH^n_{\bbC},\partial_\infty \bbH^n_{\bbC}) \ , \ \ \Phi(s):=\widehat{\phi}_s \ . 
$$
where $\Sigma$ is the space of rational functions between boundaries. By Section \ref{sec:smoothness}, we know that the joint $(\textup{SU}(n,1) \times \textup{PU}(n,1))$-action on $\Sigma$ given by 
$$
(g,h).\phi(\xi)=h\phi(g^{-1}\xi) \ , \ \ g,h \in G, \xi \in \partial_\infty \bbH^n_{\bbC} \ 
$$
is smooth. Notice that the $\sigma$-equivariance of $\phi$ implies that the map $\Phi$ satisfies
$$
\Phi(\gamma.s)=\sigma(\gamma.s)(\gamma.\Phi)(s) \ .
$$
As a consequence the map 
$$
\overline{\Phi}:\Omega \rightarrow \Sigma / \textup{PU}(n,1) \ ,
$$
is $\Gamma$-equivariant and the map 
$$
\widetilde{\Phi}:\Omega \rightarrow \Sigma / \textup{SU}(n,1) \times \textup{PU}(n,1) \ 
$$
is $\Gamma$-invariant. Since the quotient on the right is countably separated (by the smoothness of the joint action) and the space $(\Omega,\mu)$ is $\Gamma$-ergodic, the map $\widetilde{\Phi}$ is essentially constant. Equivalently the map $\overline{\Phi}$ takes values essentially in a unique $\textup{SU}(n,1)$-orbit. 

Let $s_0 \in \Omega$ be a point such that $\overline{\Phi}$ takes values in $\textup{SU}(n,1) \cdot \overline{\Phi}(s_0)$. We set 
$$
G_0:=\textup{Stab}_{\textup{SU}(n,1)}(\overline{\Phi}(s_0)) \ .
$$
The latter is algebraic by \cite[Proposition 3.3.2]{zimmer:libro} and we have a $\Gamma$-equivariant map 
$$
\overline{\Phi}: \Omega \rightarrow \textup{SU}(n,1) \cdot \overline{\Phi}(s_0) \cong \textup{SU}(n,1)/G_0 \ .
$$
As a consequence, the quotient $\textup{SU}(n,1)/G_0$ supports a $\Gamma$-invariant probability measure obtained by pushing forward the measure $\mu$ through the map $\overline{\Phi}$. By \cite[Lemma 5.1]{SS21} and by the Borel Density Theorem \cite[Theorem 3.2.5]{zimmer:libro} the group $G_0$ must coincide with $\textup{SU}(n,1)$. In this way, we deduce that $\Phi$ takes essentially values in a single $\textup{PU}(n,1)$-orbit. Equivalently, there must exist some Zariski dense rational function $\phi_0 \in \Sigma$ such that 
$$
\Phi: \Omega \rightarrow \textup{PU}(n,1) \cdot \phi_0 \cong \textup{PU}(n,1)/\textup{Stab}_{\textup{PU}(n,1)}(\phi_0) \  .
$$
By composing with a measurable section $\textup{PU}(n,1)/\textup{Stab}_{\textup{PU}(n,1)}(\phi_0) \rightarrow \textup{PU}(n,1)$ \cite[Corollary A.8]{zimmer:libro} we obtain a measurable map 
$$
\varphi:\Omega \rightarrow \textup{PU}(n,1) \ ,
$$
such that
$$
\Phi(s)=\varphi(s)\phi_0 \ .
$$
By setting
$$
\Upsilon:\Gamma \times \Omega \rightarrow \textup{PU}(n,1) \ , \ \ \ \Upsilon(\gamma,s):=\varphi(\gamma.s)^{-1}\sigma(\gamma,s)\varphi(s) \ , 
$$
the $\sigma$-equivariance of $\Phi$ implies that 
$$
\Upsilon(\gamma,s_1)\phi_0(\xi)=\phi_0(\gamma^{-1}\xi)=\Upsilon(\gamma,s_2)\phi_0(\xi) \ ,
$$
for almost every $\xi \in \partial_\infty \bbH^n_{\bbC}$ and almost every $s_1,s_2 \in \Omega$. Equivalently the product $\Upsilon(\gamma,s_1)^{-1}\Upsilon(\gamma,s_2)$ lies in the pointwise stabilizer of the essential image of $\phi_0$. By the Zariski density of $\phi_0$, the pointwise stabilizer of the essential image coincides with the pointwise stabilizer of the whole boundary $\partial_\infty \bbH^n_{\bbC}$. Since the pointwise stabilizer of $\partial_\infty \mathbb{H}^n_{\bbC}$ is trivial, we have that $\Upsilon$ does not depend on $s \in \Omega$. As a consequence we have found a morphism 
$$
\Upsilon:\Gamma \rightarrow \textup{PU}(n,1) 
$$
whose cocycle is cohomologous to $\sigma$. This proves the trivializability of $\sigma$ and concludes the proof. 
\end{proof}

\begin{oss}
Notice that the morphism $\Upsilon:\Gamma \rightarrow \textup{PU}(n,1)$ we obtained at the end of the previous proof admits a boundary map which is Zariski dense and chain preserving. As a consequence of this, $\Upsilon$ is superrigid and hence it can be extended to a morphism $\widehat{\Upsilon}$ of the ambient group. Equivalently, the cocycle we started with is trivializable to a representation $\widehat{\Upsilon}$ coming from the ambient group by restriction. 
\end{oss}

We are finally ready to prove the main theorem of the section.

\begin{proof}[Proof of Theorem \ref{teor:complex:case}]

Let $U'<W$ be a non-trivial unipotent subgroup of $W$. 

Suppose that the pair $(\mathbf{H},\kappa)$ is compatible with the group $G$. By Lemma \ref{lem:algebraic:representation} there exists an $U'$-algebraic representation $(\mathbf{J},\tau,\Psi)$ of $G \times \Omega$. By Proposition \ref{prop:j:non:trivial} we must have that $\mathbf{J}$ is trivial, otherwise $(\mathbf{H},\kappa)$ would be an incompatible pair. In this way one can conclude following the same strategy of the proof of Theorem \ref{teor:real:case}. This settles the case when $(\mathbf{H},\kappa)$ is compatible. 

Suppose now that $\kappa=\bbR$ and $\mathbf{H}(\kappa)=\textup{PU}(n,1)$. Let $(\mathbf{J},\tau,\Psi)$ be an $U'$-algebraic representation of $G \times \Omega$. By Lemma \ref{lem:same:map} we know that we can extend $\tau$ to a minimal parabolic subgroup $P$ mantaining the same map $\Psi$. 

If $\mathbf{J}$ is trivial we are done by the argument above. Thus we suppose that $\mathbf{J}$ is not trivial. We must have that $\tau:P \rightarrow (N_{\mathbf{H}}(\mathbf{J})/\mathbf{J})(\kappa)$ is incompatible in the sense of \cite{BFMS:Unpub}. By \cite[Proposition 3.4]{BFMS:Unpub} we have that $\mathbf{J}$ is the center of the unipotent radical of a minimal parabolic subgroup and $\ker(\tau)$ is one-dimensional. Thus we are in the situation of Proposition \ref{prop:rigidity:boundary:map} and the statement follows. 
\end{proof}

\section{Existence of the measurable map}\label{sec:existence:map}

Let $G$ be either $\textup{SO}^\circ(n,1)$, with $n \geq 3$, or $\textup{SU}(n,1)$, with $n \geq 2$. Consider a lattice $\Gamma < G$. So far we have seen how to trivialize cocycles assuming Zariski density and the existence of an equivariant family of projective measures. In this section we are going to provide a partial existence result for such equivariant family of measures when the lattice $\Gamma$ admits an infinite family of homogeneous $W$-invariant probability measure $(\mu_i)_{i \in \bbN}$ on $G/\Gamma$ converging to the Haar measure in the weak-$^\ast$ topology, for some $W<G$ simple non-compact subgroup. Notice in particular that the lattice must be arithmetic by the main results in \cite{BFMS21,BFMS:Unpub}. 

By \cite[Proposition 3.1]{BFMS21},\cite[Proposition 8.1]{BFMS:Unpub} we know that the existence of a sequence of probability measures $(\mu_i)_{i \in \bbN}$ as above is equivalent to the existence of an infinite family of totally geodesic submanifolds of dimension at least $2$ in the double quotient $K \backslash G /\Gamma$, where $K<G$ is a maximal compact subgroup. In the complex case we are going to assume that the submanifolds are all either real or complex, up to extracting a subsequence. By the characterization of totally geodesic submanifolds given by \cite[Lemma 3.2]{BFMS21}, \cite[Proposition 8.2]{BFMS:Unpub}, we know that there must exist a standard subgroup $W<G$ with normalizer $N=N_G(W)$, elements $g_i \in G$ and subgroups $g_iWg_i^{-1} \leq S_i \leq g_i N g_i^{-1}$ 
such that
$$
\Gamma_i:=S_i \cap \Gamma \ 
$$
is a lattice in $S_i$ and each totally geodesic submanifold has the form $S_i/\Gamma_i$.  

\begin{prop}\label{prop:equivariant:map:arithmetic}
Let $G$ be either $\textup{SO}^\circ(n,1)$, with $n\geq 3$, or $\textup{SU}(n,1)$, with $n \geq 2$. Let $\Gamma < G$ be a lattice and let $(\Omega,\mu)$ be an ergodic standard Borel probability $\Gamma$-space. Suppose that that there exist a simple non-compact $W<G$ and an infinite sequence $(\mu_i)_{i \in \bbN}$ of $W$-invariant $W$-homogeneous measures such that the Haar measure on $G/\Gamma$ is a weak-$^\ast$ limit of the sequence. Suppose that $\kappa$ is a local field and $\mathbf{H}$ is a simple connected adjoint algebraic $\kappa$-group. Denote by $H=\mathbf{H}(\kappa)$ the $\kappa$-points of $\mathbf{H}$. Consider a measurable cocycle $\sigma:\Gamma \times \Omega \rightarrow H$ so that $\sigma|_{\Gamma_i}$ is cohomologous to a cocycle whose image is contained in a proper subgroup of $\mathbf{H}$, for infinitely many $i \in \bbN$. Then there exist an irreducible representations of $\mathbf{H}$ on a $\kappa$-vector space $V$ and a probability-valued measurable map $\Phi:G \times \Omega \rightarrow \mathcal{M}^1(\mathbb{P}(V))$ which is $\sigma$-equivariant and $W$-invariant with respect to the first variable.
\end{prop}

\begin{proof}
By assumptions the cocycle $\sigma_i:=\sigma|_{\Gamma_i}$ is cohomologous to a cocycle whose image is contained in a proper subgroup $\mathbf{L}_i<\mathbf{H}$. We are allowed to pass to a subsequence and to suppose that $\dim(\mathbf{L}_i)=m$ is constant. 

We consider the $m$-th exterior power of the adjoint representation, more precisely $\wedge^m\textup{Ad}:H \rightarrow \textup{GL}(\wedge^m \mathfrak{h})$, where $\mathfrak{h}=\textup{Lie}(H)$ is the Lie algebra associated to $H$. In a similar way we define $\mathfrak{l}_i=\textup{Lie}(\mathbf{L}_i(\kappa))$ the Lie algebra of $\mathbf{L}_i(\kappa)$. Since we supposed that $\dim(\mathbf{L}_i)=m$, we have that $\mathfrak{l}_i$ determines a line $\ell_i$ in $\wedge^m \mathfrak{h}$. Additionally the line $\ell_i$ cannot be $H$-invariant by the properness of the normalizer of $\mathbf{L}_i$ (being $\mathbf{H}$ $\kappa$-simple). Up to extracting another subsequence, we can suppose that the line $\ell_i$ projects non-trivially on a fixed irreducible summand $V$ of the representation $\wedge^m \textup{Ad}$. Notice that the stabilizer of $\ell_i$ contains $\mathbf{L}_i(\kappa)$ by construction, so it contains the image of a cocycle cohomologous to $\sigma_i$. Equivalently there exists a measurable function $\mathscr{L}_i:\Omega \rightarrow \bbP(V)$ which is $\sigma_i$-equivariant, that is $\mathscr{L}_i(\gamma.s)=\sigma_i(\gamma,s)\mathscr{L}_i(s)$ for every $\gamma \in \Gamma_i$ and almost every $s \in \Omega$.

Consider now the product $G \times \Omega \times \bbP(V)$ with the $\Gamma_i$-action given by
$$
\gamma.(g,s,\ell)=(\gamma g,\gamma.s,\sigma_i(\gamma,s)\ell) \ .
$$
The existence of a measurable $\sigma_i$-equivariant map $\mathscr{L}_i:\Omega \rightarrow \bbP(V)$ implies the existence of a measurable section
$$
\lambda_i:S_i/\Gamma_i \times \Omega  \rightarrow (G \times \Omega \times \bbP(V))/\Gamma_i \ .
$$
Here we retained the same notation we used at the beginning of the section.

We can set  $\nu_i:=(\lambda_i)_\ast (\mu_i \otimes \mu)$. We consider $\nu$ an ergodic component of the weak-$\ast$ limit of the sequence $(\nu_i)_{i \in \bbN}$ (the existence of such limit is guaranteed by the Banach-Alaoglu Theorem). Since weak-$^\ast$ limit and projections commute, the weak-$^\ast$ convergence of $\mu_i$ to the Haar measure $\mu_G$ on $G/\Gamma$ implies that $\nu$ projects to the product measure $\mu_G \otimes \mu$ on $G/\Gamma \times \Omega$. 

Now we can disintegrate $\nu$ using \cite[Theorem 2.1]{Hah78}. We obtain a measurable map 
$$
\varphi:\Omega \rightarrow \mathcal{M}^1(G \times \bbP(V)) \ .
$$
For almost every $s \in \Omega$, we can disintegrate again the measure $\varphi(s)$ obtaining a measurable map 
$$
\Phi_s:G \rightarrow \mathcal{M}^1(\bbP(V)) \ .
$$
Collecting together all the measurable maps $(\Phi_s)_{s \in \Omega}$, we obtain a measurable map 
$$
\Phi: G \times \Omega \rightarrow \mathcal{M}^1(\bbP(V)) \ .
$$
Additionally the $\Gamma$-invariance of the measure $\nu$ and the uniqueness of the disintegration, implies that $\Phi$ is $\sigma$-equivariant. The $W$-invariance of $\nu$ tells us that $\Phi$ is $W$-invariant on the first variable. This proves the statement and concludes the proof. 
\end{proof}

\section{Final remarks and comments}\label{sec:comments:remarks}

We want to conclude this short manuscript trying to relate our results with the current literature about superrigidity of measurable cocycle for hyperbolic lattices. 

We consider the complex case. Given a measurable cocycle $\Gamma \times \Omega \rightarrow \textup{PU}(n,1)$, where $\Gamma<\textup{SU}(n,1)$ with $n\geq2$, we know by \cite[Theorem 1.5]{moraschini:savini:2} that if the cocycle is maximal it has to be trivializable. Theorem \ref{teor:complex:case} offers us another point of view. In fact, if we forget about maximality and we assume Zariski density, the existence of an equivariant family of projective measures (for instance as in Proposition \ref{prop:equivariant:map:arithmetic}) implies the trivializability of the cocycle. Notice that, \emph{a posteriori}, maximal cocycles are Zariski dense, so in principle Theorem \ref{teor:complex:case} cover a wider family of cases than maximal cocycles (when the target is $\textup{PU}(n,1)$). More generally, we believe that Theorem \ref{teor:complex:case} gives a contribution in understanding which Zariski dense cocycles are trivializable if we do not want to invoke maximality to trivialize them.  

Theorems \ref{teor:real:case} and \ref{teor:complex:case} help also to understand better how to weaken the hypothesis of the results by Fisher-Hitchman \cite{fisher:hitchman} and by Bader-Furman-Sauer \cite{sauer:articolo}, where some integrability assumptions are required to get the desired superrigidity statements.

\bibliographystyle{amsalpha}

\bibliography{biblionote}

\end{document}